\documentclass[twoside]{article}

\usepackage{lipsum} 
\usepackage{multirow}
\usepackage[sc]{mathpazo} 
\usepackage[T1]{fontenc} 
\usepackage{microtype} 

\usepackage[hmarginratio=1:1,top=32mm,columnsep=20pt]{geometry} 
\usepackage{multicol} 
\usepackage[hang, small,labelfont=bf,up,textfont=it,up]{caption} 
\usepackage{booktabs} 
\usepackage{float} 
\usepackage{hyperref} 

\usepackage{lettrine} 
\usepackage{paralist} 

\usepackage{abstract} 

\usepackage{titlesec} 
\renewcommand\thesection{\Roman{section}} 
\renewcommand\thesubsection{\Roman{subsection}} 
\titleformat{\section}[block]{\large\scshape\centering}{\thesection.}{1em}{} 
\titleformat{\subsection}[block]{\large}{\thesubsection.}{1em}{} 

\usepackage{fancyhdr} 
\usepackage[pdftex]{graphicx}
\usepackage[square, numbers, comma, sort&compress]{natbib}
\usepackage{algorithm}
\usepackage{algorithmic}
\usepackage{amsmath}
\usepackage{amssymb}
\usepackage{amsthm}
\usepackage{color, colortbl}
\usepackage{hhline}
\usepackage{hyperref}
\usepackage{multicol}
\usepackage{fancyhdr}
\usepackage{graphicx}
\usepackage{enumerate}
\usepackage{hyperref}
\usepackage[table]{xcolor}
\usepackage{graphicx}
\usepackage{rotating}
\usepackage{multirow}
\usepackage{wrapfig}
\usepackage{caption}
\usepackage{subcaption}
\usepackage{url}
\usepackage{textcomp}
\usepackage{tipa}
\usepackage{verbatim}
\usepackage{framed}
\usepackage{slashbox}

\hypersetup{draft}

\newcommand{\rset}{\mathbb{R}}

\title{\vspace{-15mm}\fontsize{24pt}{10pt}\selectfont\textbf{Maximizing graph probability under conditionally exponential models}} 

\author{
\large
\textsc{Stefano Nasini}\\[2mm] 
\normalsize Universitat Politecnica de Catalunya \\ 
\normalsize \href{stefano.nasini@upc.edu}{stefano.nasini@upc.edu} 
\vspace{-5mm}
}
\date{}

\begin{document}

\maketitle 

\pagenumbering{gobble}
\thispagestyle{empty} 


\begin{abstract}

Designing reliable networks consists in finding topological structures, which are able to successfully carry out desired processes and operations. When this set of activities performed within a network are unknown and the only available information is a probabilistic model reflecting topological network features, highly probable networks are regarded as "reliable", in the sense of being consistent with those probabilistic model. In this paper we are studying the reliability maximization, based on the Exponential Random Graph Model (ERGM), whose statistical properties has been widely used to capture complex topological feature of real-world networks. Under such models the probability of a network is maximized when specified structural properties appear in the network. However, the search of maximally reliable (highly probable) networks might result in difficult combinatorial optimization problems and an important goal of this work is to translate them into solvable systems of linear constraints. Analytical and numerical results are provided, using exact optimization techniques and efficient computer implementation.

\end{abstract}

\hspace{0.4cm}\textbf{Key words}: Network reliability; Random graphs; Combinatorial Optimization.


\begin{multicols}{2} 

\section{Introduction}

\lettrine[nindent=0em,lines=3]{N}etwork reliability has been studied as the ability of a network to successfully carry out desired processes and operations. When those processes are only probabilistically known, the design of reliable networks might consist in the maximization of the probability of success of a given operation within a network or the minimization of its failure.

Sometimes the only available information is a probability distribution on the set of edges, so that the optimal decision of a global planner consists in the search of highly probable networks, under the specified probabilistic model.

The first and best known probabilistic model of network formation has been introduced by the seminal work of Erd\"os and Rainyi \cite{Erdos1959}, who considered a fixed set of $n$ nodes and an independent and equal probability $2d/(n(n-1))$ of observing edges among them. A related variant of this model considers networks chosen uniformly at random from the collection of all graphs with $n$ nodes and $d$ edges. (For more details about network properties, see Bollobas \cite{Bollobas1985}, and Wasserman and Faust \cite{WassermanFaust1994}.)

This variant of the Erd\"os--Rainyi random graph might be regarded as a conditionally uniform random model and a straightforward generalization of it is to consider different conditioning information, other than the number of edges. A general approach to characterize families of networks with fixed conditioning information is based on generating integer solution of systems of linear constraints \cite{CastroNasini2014}.

Let $x_{ij}$ be the entries of the adjacency matrix (AM, from now on) of an undirected graph and $\mathbf{x}$ the vector associated to its components sorted in lexicographic order. Castro and Nasini \cite{CastroNasini2014} considered the probability space $(\chi, p, \Im)$, where $\chi$ is the set of all AMs of networks verifying a specified collection of linear constraints, that is to say $\chi = \{ x_{ij} \in \{0,1\}, (i,j) \in \mathcal{H}^2 ~:~ A\mathbf{x}=\mathbf{b} \}$ and $\mathcal{H}^2 = \{(i,j) : 1\le i \le n-1, i < j \leq n\}$. The set $\Im$ is a $\sigma$-algebra on $\chi$ and $p$ a probability measure. In the case of dealing with the set of networks with fixed number of edges, we have $A = [1 \ldots 1]$.

The probability measure $p$ might take different forms, capturing the processes and operations performed within the network. Exponential Random Graph Models (ERGM) has been widely used to capture the fact that the probability of a network should reflect its structural properties. Here we are considering Conditionally Exponential Random Graph Models (CERGOM), which keeps the main probabilistic properties of the classical ERGM into a constrained sample space $\chi$, as shown in the next section. Highly probable networks are regarded as "reliable" in the sense of being consistent with the real-world scenario captured by the specified probabilistic model, so that the problem of designing reliable networks is here translated into the one of maximizing graph probability under conditionally exponential models.

This analysis casts a mathematical bridge between probabilistic and optimization based models of network formation, which allow takeing into account the emerging properties from the point of a global planner, who wish to allocate connections among nodes in such a way as to successfully carry out  processes and operations performed within the networks.

The rest of this paper provides a mathematical programming based approach to deal with this problem, based on the linearization of complex combinatorial properties.


\section{Reliable networks by probability maximization}


Consider a collection of independent and identically distributed networks $\mathbf{x}_1, \ldots, \mathbf{x}_N \sim p$. Let $S_j(\mathbf{x}_i)$, for $j = 1, \ldots, s$, be a structural feature of network $\mathbf{x}_i$, for $i = 1, \ldots, N$, and $\widehat{\mu}_j = \sum_{i=1}^{N} S_j(\mathbf{x}_i)/N$ the empirical expectation of $S_j$, for $j = 1, \ldots, s$. The ERGM arises as an answer to the question \emph{can we recover} $p$ \emph{from} $\widehat{\mu}_1, \ldots, \widehat{\mu}_s $? A reasonable requirement a probability measure $p$ must verify is that $\mathbb{E}_p[S_j(\mathbf{x})] = \int_{\chi} S_j(\mathbf{x})p(\mathbf{x}) dx = \widehat{\mu}_j$, for $f = 1 \ldots s$, where $\chi = \{ x_{ij} \in \{0,1\}, (i,j) \in \mathcal{H}^2 ~:~ A\mathbf{x}=\mathbf{b} \}$.
(We are using the Lebesgue integral as a generalization of summation for continuous spaces, that is to say, spaces of valued networks, where $\chi$ is the set of valued AMs verifying a specified system linear constraints. In the cases of zero-one-networks considered in this paper the integral can be replaced by a summation.)

To solve this problem, the maximum entropy principle leads to pick the distribution $\widehat{p(\mathbf{x})}$, which verifies $\mathbb{E}_p[S_j(\mathbf{x})] = \widehat{\mu}_j$, while maximizing the entropy:
\begin{subequations}\label{eq:1.3.2_2}
\begin{eqnarray*}
\max& & -\displaystyle \int_{\chi} p(\mathbf{x}) \log p(\mathbf{x}) dx\\
\hbox{s. to}
& & \int_{\chi} S_j(\mathbf{x})p(\mathbf{x}) dx = \widehat{\mu}_j, \quad f = 1, \ldots, s\\
& & \int_{\chi} p(\mathbf{x}) dx = 1.
\end{eqnarray*}
\end{subequations}

The classical method for solving this class of problems is to apply the Lagrange multipliers to each of the constraints and maximize the augmented functional with respect to $p(\mathbf{x})$.
As described by Jaynes \cite{Jaynes1957}, applying the Euler equation of calculus of variation, the functional form of a conditionally exponential random graph model might be obtained in few algebraical steps:
\begin{equation*} \label{eq:1.3.2_4}
\widehat{p(\mathbf{x})} = \left\{
   \begin{array}{l l}
     \displaystyle \kappa \exp \left( \boldsymbol{\theta} \mathbf{S}(\textbf{x}) \right) & \quad \text{if $\mathbf{x} \in \chi$}\\
     0 & \quad \text{otherwise}
   \end{array} \right.
\end{equation*}
The boldface symbols $\boldsymbol{\theta}$ and $\mathbf{S}(\mathbf{x})$ denote the vectors $[\theta_1 , \ldots, \theta_s]^T$ and $[S_1(\mathbf{x}) , \ldots, S_s(\mathbf{x})]^T$ respectively. The vector parameter $\theta_j \in \rset^{s}$ controls the tendency of networks with parameters $S_j(\mathbf{x})$ to be observed in the data. The scalar quantity $\kappa = \int_{\textbf{x} \in \boldsymbol{\chi} }e^{\sum_{j = 1}^{s} \theta_j S_j(\mathbf{x})} \, d\textbf{x}$ is known as the partition function. Many papers from statistical physics call the quantity $\sum_{j = 1}^{s} \theta_j S_j(\mathbf{x})$ the \emph{graph Hamiltonian} and denote it as $H(\mathbf{x})$.

The just described CERGM represents a generalization of the classical ERGM, where the sample space $\chi$ can be arbitrarily defined. If $\chi$ is the set of all simple graphs, then $|\chi|= 2^{\binom{n}{2}}$. If $\boldsymbol{\chi}$ is the set of all undirected graphs with fixed number of edges $d$, then $|\chi|= \binom{n(n-1)/2}{d}$.

Finding a reliable network under the specified CERGM consist in maximizing $\kappa \exp( \boldsymbol{\theta} \mathbf{S}(\textbf{x}))$, subject to $\mathbf{x} \in \chi$.

The highly combinatorial nature of these problems is analyzed in the next section, based on different mathematical programming formulations.


\section{Network properties by systems of linear constraints}


Consider the set $\chi$ of undirected networks with fixed number of edges $d$ and an exponential model on the sample space $\chi$, with graph Hamiltonian $H(\mathbf{x}) = \theta \sum_{i < j < k} x_{ij} x_{jk} x_{ij}$, representing the number of closed triangles in the network. Maximizing the logarithm of the probability inside the specified sample space leads to the non-linear binary problem
\begin{equation}\label{eq:MaxTriads_FixedDensity}
\begin{array}{rl}
\max         & \displaystyle \sum_{i < j < k} x_{ij} x_{jk} x_{ij}\\
\hbox{s. to} & \displaystyle \sum_{i < j} x_{ij} = d \\
             & \displaystyle x_{ij} \in \{0,1\} ~ (i,j) \in \mathcal{H}^2.
\end{array}
\end{equation}
Let $w_{ijk}$ be the binary indicator of the $(i,j,k)$ triad, which is equal to one if a closed triangle between $i$, $j$ and $k$ exists, and zero otherwise. The following system of linear constraints characterize the state of $w_{ijk}$:
\begin{equation}\label{eq:triad_B}
\begin{array}{rl}
                  1 - z_{ijk} ~ \leq ~ x_{ij} ~ \leq ~ y_{ijk}               &  \\
                  1 - z_{ijk} ~ \leq ~ x_{jk} ~ \leq ~ y_{ijk}               &  \\
                  1 - z_{ijk} ~ \leq ~ x_{ij} ~ \leq ~ y_{ijk}               &  \\
                  y_{ijk} - z_{ijk} ~ \leq ~ w_{ijk}                         &  \\
                  w_{ijk} + z_{ijk} ~ \leq ~ 1                               &  \\
                  3 - z_{ijk} \geq x_{ij} + x_{jk} + x_{ik}                  &  \\
                  y_{ijk}, z_{ijk}, w_{ijk}  \in \{0,1\}                     &
\end{array}
\end{equation}
where $y_{ijk}$ and $z_{ijk}$ are auxiliary variables.

By introducing (\ref{eq:triad_B}) in (\ref{eq:MaxTriads_FixedDensity}), for every $(i,j,k)$ such that $1\le i \le n-2, ~ i < j \leq n-1, ~ j < k \leq n$, and replacing the objective function $\sum_{i < j < k} x_{ij} x_{jk} x_{ij}$ with $\sum_{(i,j,k) \in \mathcal{H}^{3}} w_{ijk}$, the maximization of the graph probability under the specified model becomes a linear program in binary variables, which can be solved up to optimality by standard integer programming technics \cite{Schrijver1998, Schrijver2003}.

Note that for particularly small $d$ (high sparsity) an optimal network consist of a fully connected subgraph and several disconnected nodes. One possibility to overcome this drawback is either to include more network properties in the graph Hamiltonian or to redefine the sample space $\chi$ in such a way that the unwanted trivial solutions are discarded, for example by forcing \emph{network connectivity}.

A straightforward way to algebraically force connectivity is to require the existence of a flow circulating within the network, from one node to all the others. Thus, we make use of an artificial flow of $n-1$ units, departing from one node $h \in \mathcal{V}$ and arriving to each of the $n-1$ remaining nodes. The existence of such flow is a sufficient and necessary condition for the network to be connected:
\begin{equation}\label{eq:connectivity}
\begin{array}{rll}
                  \displaystyle \sum_{j = 1}^{n} f_{hj} - \displaystyle \sum_{j = 1}^{n} f_{jh} &= n-1    & \\
                  \displaystyle \sum_{j = 1}^{n} f_{kj} - \displaystyle \sum_{j = 1}^{n} f_{jk} &= -1     & k \neq h \\
                  f_{ij} + f_{ji} & \leq n x_{ij}            & \multirow{2}{*}{$(i,j) \in \mathcal{H}^{2}$}  \\
                  f_{ij} f_{ji}   \geq 0 &                   &

\end{array}
\end{equation}

Based on the same network flow intuition the \emph{average path length}, which is the average number of steps along the shortest paths for all pairs of nodes --$\frac{1}{n(n-1)} \sum_{i\neq j} \delta_{ij}$, where $\delta(i,j)$ is the shortest path between $i$ and $j$ -- can be linearly modeled and introduced in the graph Hamiltonian.

To measure the minimum amount of circulating flow needed to deliver one unit of flow from one node to another, $n$ types of commodity-flows are defined (one per each node) along with $n-1$ units of each type of flow departing from the corresponding node and reaching all the remaining nodes:
\begin{equation}\label{eq:FlowBased_APL}
\begin{array}{rlll}
   \displaystyle \sum_{k = 1}^{n} f_{ij}^{k} + f_{ji}^{k} \leq n^{2} x_{ij},                        ~ (i,j) \in \mathcal{H}^{2}  \\
   \displaystyle \sum_{j = 1}^{n} f_{hj}^{h} - \displaystyle \sum_{j = 1}^{n} f_{jh}^{h} = n-1,  ~ h \in \mathcal{V} \\
   \displaystyle \sum_{j = 1}^{n} f_{kj}^{h} - \displaystyle \sum_{j = 1}^{n} f_{jk}^{h} = -1,   ~ k \neq h \\
   f_{ij} f_{ji} \geq 0,  ~ (i,j) \in \mathcal{H}^{2}
\end{array}
\end{equation}
The total flow-based distance between nodes is then defined as $S(\textbf{x}) = \min_{f} \sum_{(i,j) \in \mathcal{H}^{2}} \sum_{k = 1}^{n} f_{ij}^{k} + f_{ji}^{k}$, subject to (\ref{eq:FlowBased_APL}).

Another commonly used network parameter is the correlation between nodes properties or the physical distance between nodes. Consider a measure of distance between nodal properties $\delta$. The total distance between nodes in the networks with respect to the space generated by the nodal properties is $S(\textbf{x}) = \sum_{(i,j) \in \mathcal{H}^{2}} \delta_{ij} x_{ij}$.

Under the exponential family, network statistics must be additivity included in the graph Hamiltonian, as resulted from the maximum entropy principle: $\kappa \exp ( \boldsymbol{\theta} \mathbf{S}(\textbf{x}) )$. Nonetheless, it has been widely discussed (see Handcook \cite{Handcock2003}) that this probabilistic models suffer of high \emph{degeneracy} and lack of \emph{robustness} (or \emph{stability}). Degeneracy entails that the model places a disproportionate probability mass on only on a few of the possible graph configurations in $\chi$. The lack of robustness can be seen from two different point of view:
\begin{itemize}
\item small changes in the parametrization $\boldsymbol{\theta}$ give rise to drastic changes in the optimal networks;
\item several alternative optimal solutions with drastically different topological structures exist.
\end{itemize}

To deal with the robustness of the optimal solution, consider the maximal graph probability $P^{*} = \max H(\mathbf{x})$, s. to $\mathbf{x} \in \chi$, and suppose there exist multiple $\mathbf{x} \in \chi$ with the same optimal value $P^{*}$. A robust solution of the problem of maximizing graph probability under conditionally exponential models is given by a second stage max--min formulation, which use the previously found optimal solution as a lower bound:
\begin{equation}\label{eq:MaxMinProb}
\begin{array}{rl}
\max          & H\\
              & H \leq \theta_j S_j(\mathbf{x}) \qquad j = 1 \ldots s\\
 \hbox{s. to} & \displaystyle \sum_{j = 1}^{s} \theta_j S_j(\mathbf{x}) \geq \gamma P^{*}\\
              & \displaystyle \mathbf{x} \in \chi,
\end{array}
\end{equation}
where $\gamma \in [0,1]$ is a tuning parameter to allow for sub-optimality. When $\gamma = 1$, the optimal solution of $(\ref{eq:MaxMinProb})$ is a network which maximizes the conditionally exponential probability, while ensuring max--min decision criterion (as a robustness criterion) among the alternative optimal networks. When $\gamma = 0$, the optimal solution of $(\ref{eq:MaxMinProb})$ is a network which maximizes the graph probability based on a \emph{non-linear-graph-Hamiltonian} $\kappa \exp ( \min \{\theta_i S(\textbf{x}) ~:~ i = 1 \ldots s \} )$, which doesn't belong to the exponential family.

The next section will numerically study the problems with the form of $(\ref{eq:MaxMinProb})$, with $\gamma = 0$.


\section{Max-min formulation and numerical results}


Let $S_1(\mathbf{x})$ and $S_2(\mathbf{x})$ be the number of non-connected pairs of nodes ($n(n-1)/2 - \sum_{(i,j) \in \mathcal{H}^{2}} x_{ij}$) and number of triangles (characterized by (\ref{eq:triad_B}) ) of a network $\mathbf{x}$. We wish to maximize the graph probability within the sample space of all connected networks (characterized by (\ref{eq:connectivity}) ) and under a non-linear graph Hamiltonian with the form of $H(\textbf{x}) = \min \{\theta_i S(\textbf{x}) ~:~ i = 1 \ldots s \}$. The resulting max--min formulation is the following:
\begin{equation}\label{eq:MaxMin}
\begin{array}{rll}
\max              & H & \\
\mbox{s. to}      & H \leq \alpha S_1(\mathbf{x})  \\
                  & H \leq (1-\alpha) S_2(\mathbf{x}) \\
                  & \mathbf{x} \in \chi,
\end{array}
\end{equation}
where $\alpha = \frac{\theta_1}{\theta_1+\theta_2}$ -- a rescaling of the model parameters --.

The direct computation of the optimal solution of problem (\ref{eq:MaxMin}) by brunch and bound algorithm is time consuming when $n \geq 50$ and a strong lower bound of the optimal solution can be used to reduce the number of brunch and bound nodes.
\newtheorem{proposition}{Proposition}
\begin{proposition}
The number of closed triangles in the optimal solution of problem (\ref{eq:MaxMin}) is bounded from below by $h = \min \left\{ (n-1), ~ \alpha \frac{(n-2)(n-1)}{2}\right\}$ and the number of edges is bounded from below by $(n-1) + h$.
\end{proposition}

\begin{proof}
Since $\chi$ is the set of all connected graphs, we must have $\sum_{(i,j) \in \mathcal{H}^{2}} x_{ij} \geq n-1$. Every connected graph with $n-1$ edges -- a tree -- is associated to $H = 0$ in (\ref{eq:MaxMin}), as no closed triangle exist in it. Consider a star as a feasible solution of (\ref{eq:MaxMin}) with $H = 0$ and add $h$ edges between pairs of nodes at distance two. Since we have $n-1$ pairs of nodes at distance two in a star, $h \leq n-1$ additional edges might be included, creating $h$ new triangle and resulting in an increase of the value of the objective function up to
$\min \left\{ (1-\alpha) h, ~ \alpha \left(\frac{n(n-1)}{2} - (n-1) - h \right) \right\}$.
Thus, we can keep increasing $h$ as long as $(1-\alpha) h \leq \alpha \left(\frac{n(n-1)}{2} - (n-1) - h \right)$, up to $n-1$, so that the number of closed triangles and the number of edges in the optimal solution will be bounded from below by $h = \min \left\{ (n-1), ~ \alpha \frac{(n-2)(n-1)}{2}\right\}$ and $(n-1) + h$, respectively.
\end{proof}

Stronger valid inequalities can be numerically found by heuristic methods, such as local search, tabu search and ant colony\footnote{We implemented a first-improve local search which adds and remove one edge in each iteration up to converge to a local optimum, in which no single edge can be added or removed with an improve of the objective function.}.

Three instances of problem (\ref{eq:MaxMin}) are solved for $n = 60$ and $\alpha = 0.7$, $0.5$, $0.3$, as shown in figures \ref{fig:MaxMin_TriadsDensity_1}, \ref{fig:MaxMin_TriadsDensity_2} and \ref{fig:MaxMin_TriadsDensity_3} respectively.

\begin{center}
\begin{figure}[H]
        \includegraphics[height=60mm, width=65mm]{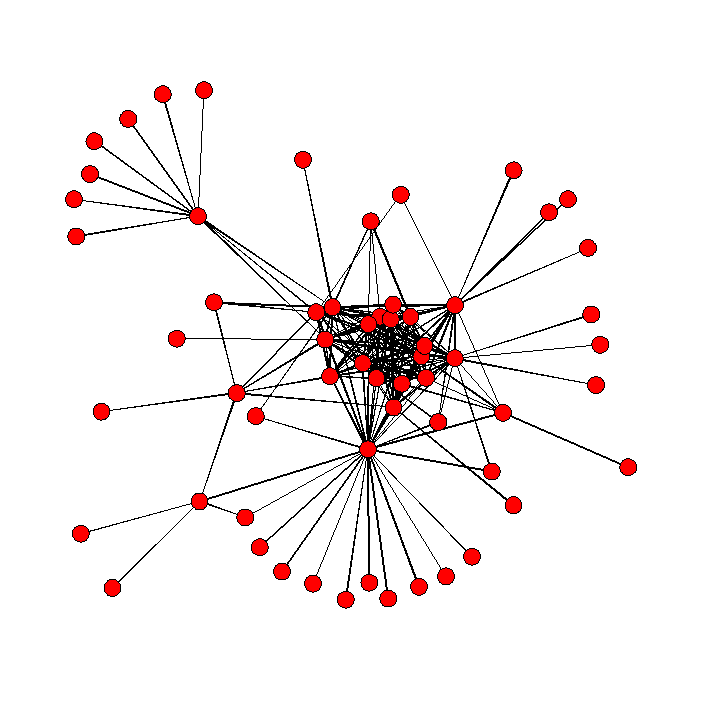}
\caption{\label{fig:MaxMin_TriadsDensity_1} \footnotesize Optimal network, for $n=60$ and $\alpha = 0.7$.}
\end{figure}
\end{center}

\begin{center}
\begin{figure}[H]
        \includegraphics[height=60mm, width=65mm]{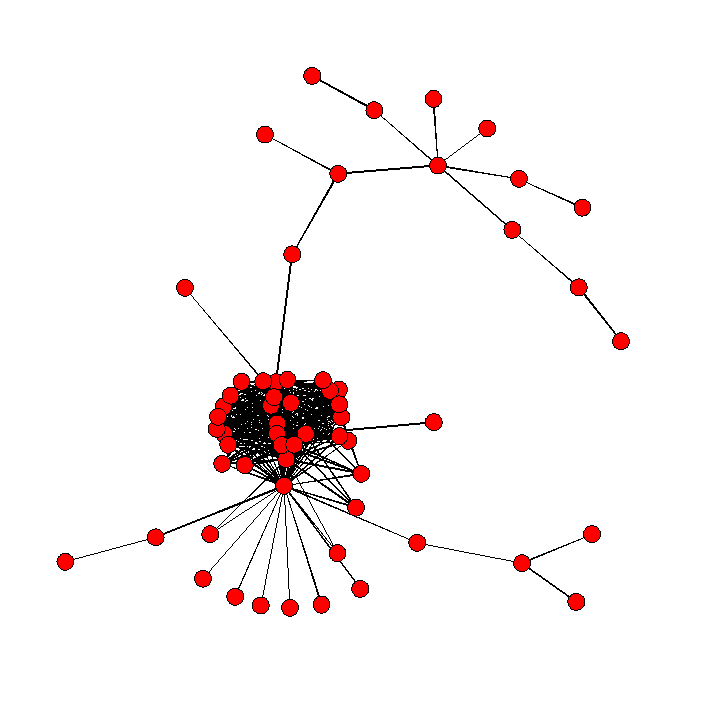}
\caption{\label{fig:MaxMin_TriadsDensity_2} \footnotesize Optimal network, for $n=60$ and $\alpha = 0.5$.}
\end{figure}
\end{center}

\begin{center}
\begin{figure}[H]
        \includegraphics[height=60mm, width=65mm]{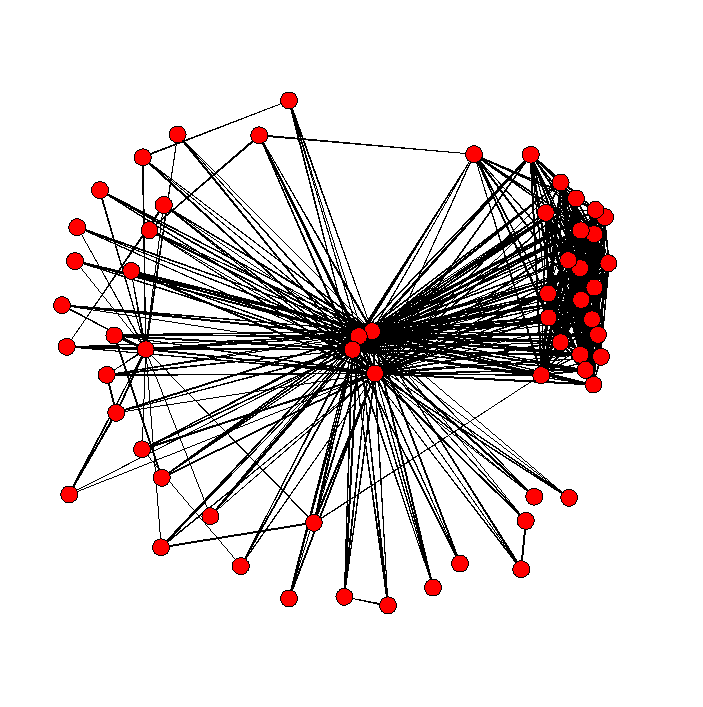}
\caption{\label{fig:MaxMin_TriadsDensity_3} \footnotesize Optimal network, for $n=60$ and $\alpha = 0.3$.}
\end{figure}
\end{center}

Table \ref{tab:MaxMin_TriadsDensity} reports the density, clustering coefficient and average path lengths of the three networks shown in figure \ref{fig:MaxMin_TriadsDensity_1}, \ref{fig:MaxMin_TriadsDensity_2} and \ref{fig:MaxMin_TriadsDensity_3}.

\begin{table}[H]
\centering{}%
\scalebox{0.95}{
\begin{tabular}{lrrr}
\hline
Dataset               & Density & CC      & APL       \tabularnewline
\hline
$\alpha = 0.7$        & 0.11921 & 0.61378 & 2.49265   \tabularnewline
$\alpha = 0.5$        & 0.17966 & 0.66311 & 3.17344   \tabularnewline
$\alpha = 0.3$        & 0.28192 & 0.63821 & 1.71808   \tabularnewline
\hline
\end{tabular}}\caption{\label{tab:MaxMin_TriadsDensity} {\footnotesize{Resulting topological properties.}}}
\end{table}

Note that, although the average path length was not maximized, a really small degree of separation between nodes emerges, as shown in Table \ref{tab:MaxMin_TriadsDensity}.

To obtain a deeper understanding about the emergence of a small degree of separation between nodes, consider an alternative model, where two measures of nodal distance are introduced:
\begin{itemize}
\item the physical distance $S_1(\textbf{x}) = \sum_{(i,j) \in \mathcal{H}^{2}} \delta_{ij} x_{ij}$, where $\delta_{ij}$ is the distance between $i$ and $j$;
\item the network distance $S_2(\textbf{x})$, measured by the size of the total circulating flow to deliver one unit of flow from one node to another, as formulated in (\ref{eq:FlowBased_APL}).
\end{itemize}

The max--min formulation in (\ref{eq:MaxMinProb}) with $\gamma=0$ is considered\footnote{In the case of dealing with negative model parameters $\theta_1$ and $\theta_2$ associated to the two types of distances the max--min formulation in (\ref{eq:MaxMinProb}) can be replaced by a min--max formulation, where the signs of $\theta_1$ and $\theta_2$ are switched to positive and then rescaled as $\alpha = \frac{\theta_1}{\theta_1+\theta_2}$.} and three instances are solved for $n = 60$ and $\alpha = 0.7$, $0.5$, $0.3$, as shown in figures \ref{fig:MinMax_ConnectivityWiring_1}, \ref{fig:MinMax_ConnectivityWiring_2} and \ref{fig:MinMax_ConnectivityWiring_3} respectively\footnote{This problem allows a straightforward application of the Benders decomposition, which separately and iteratively solves the problem of finding the minimum cost network structure and the one of setting the minimum circulating flow within it. We solve it by the brunch and bound algorithm, providing a lower bound to the optimal solution.}.

\begin{center}
\begin{figure}[H]
        \includegraphics[height=60mm, width=65mm]{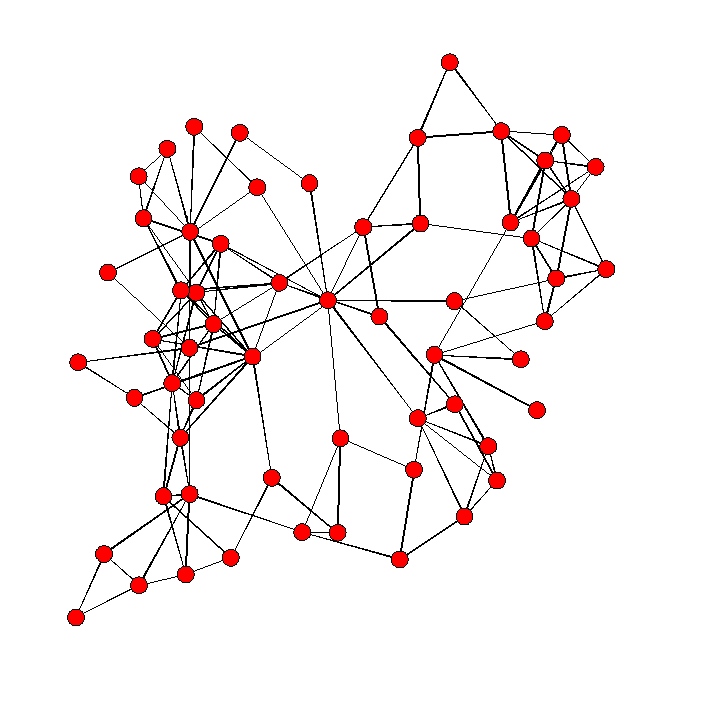}
\caption{\label{fig:MinMax_ConnectivityWiring_1} \footnotesize Optimal network, for $n=60$ and $\alpha = 0.1$.}
\end{figure}
\end{center}

\begin{center}
\begin{figure}[H]
        \includegraphics[height=60mm, width=65mm]{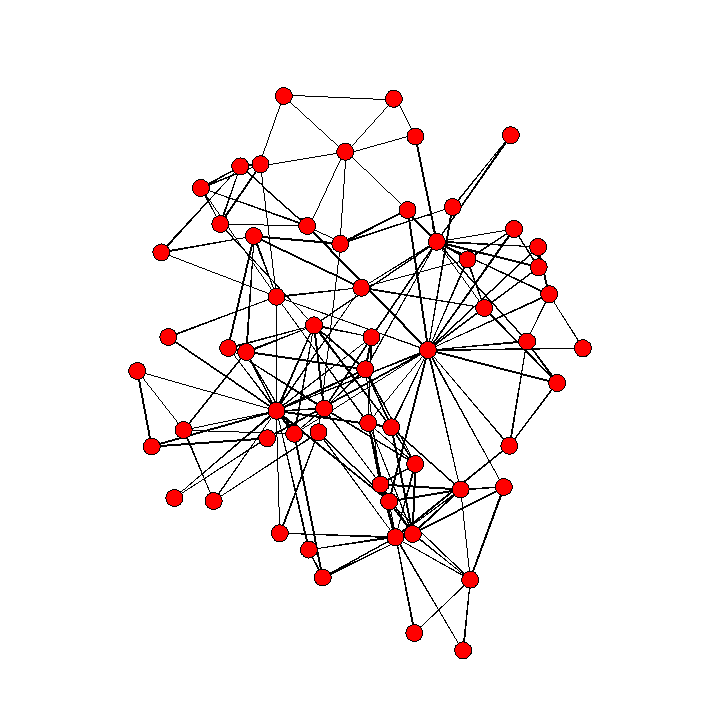}
\caption{\label{fig:MinMax_ConnectivityWiring_2} \footnotesize Optimal network, for $n=60$ and $\alpha = 0.2$.}
\end{figure}
\end{center}

\begin{center}
\begin{figure}[H]
        \includegraphics[height=60mm, width=65mm]{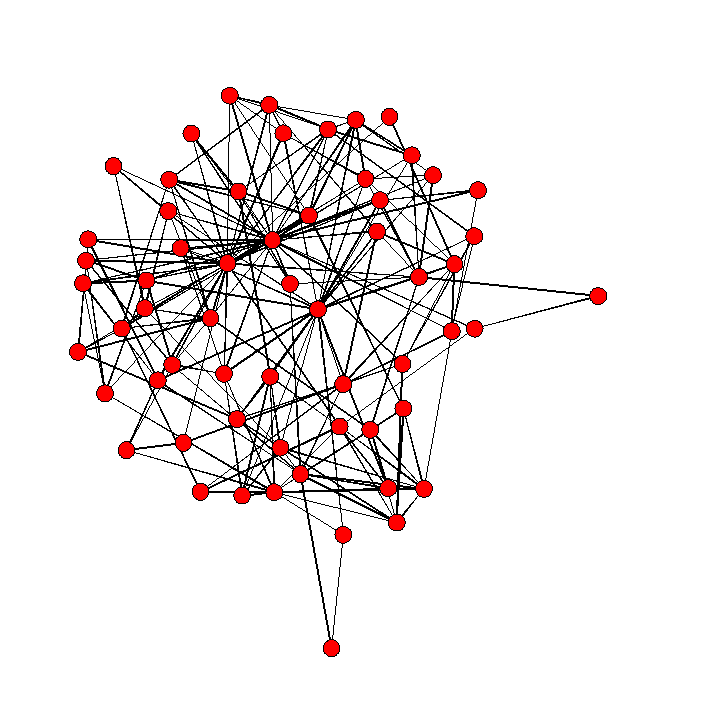}
\caption{\label{fig:MinMax_ConnectivityWiring_3} \footnotesize Optimal network, for $n=60$ and $\alpha = 0.3$.}\label{fig:GG_3}
\end{figure}
\end{center}

Although we only require to minimize the maximum between circulating total flows and total physical distance between connected nodes, the resulting networks shown in figure \ref{fig:MinMax_ConnectivityWiring_1}, \ref{fig:MinMax_ConnectivityWiring_2} and \ref{fig:MinMax_ConnectivityWiring_3} exhibit a remarkably high transitivity (high clustering coefficient), as shown in Table \ref{tab:MinMax_ConnectivityWiring}.

\begin{table}[H]
\centering{}%
\scalebox{0.95}{
\begin{tabular}{lrrr}
\hline
Dataset               & Density & CC      & APL       \tabularnewline
\hline
$\alpha = 0.1$        & 0.07910 & 0.41945 & 3.24633   \tabularnewline
$\alpha = 0.2$        & 0.09944 & 0.27842 & 2.49380   \tabularnewline
$\alpha = 0.3$        & 0.28192 & 0.23901 & 2.16554   \tabularnewline
\hline
\end{tabular}}\caption{\label{tab:MinMax_ConnectivityWiring} {\footnotesize{Resulting topological properties.}}}
\end{table}

This result seems to be particularly coherent with the one of Mathias and Gopa \cite{MathiasGopa2001}, who showed that the small-world topology arises as a consequence of a tradeoff between maximal connectivity and minimal wiring.

\small
\begin{quotation}
\hspace{-6mm} \emph{Consider a toy model of the brain. Let us assume that it consists of local processing units, connected by wires. What constraints act on this system? On the one hand, one would want the highest connectivity (shortest path length) between the local processing units, so that information can be exchanged as fast as possible. On the other hand, it is wasteful to wire everything to everything else.} (See Mathias and Gopa \cite{MathiasGopa2001}, page 2.)
\end{quotation}
\normalsize

The model of Mathias and Gopa tries to merge the intuition of a spatial distribution of nodes with the optimization based criterion of edge formation. Our approach in this context is purely based on the optimization of graph probability under a specified random model which reproduce the pattern of high transitivity and high connectivity of small-world network structures.


\section{Conclusion}

This work provided mathematical programming based approaches to deal with the problem of maximizing graph probability under conditionally exponential models. The underlying idea was to obtain highly reliable networks, when the set of operation and processes performed within a network are unknown and the only available information is a probabilistic model reflecting topological network features.


We have seen that the search of maximally reliable (highly probable) networks might result in difficult combinatorial optimization problems and an important goal of this work was to translate them into solvable systems of linear constraints.

The ability of characterizing the sample space and the graph Hamiltonian by a well defined systems of linear constraints represented a necessary (and sufficient, due to their solvability by integer programming technics) condition to generate maximally reliable networks under a conditionally exponential model, as shown in the described computational results.



\end{multicols}

\end{document}